\documentclass[12pt, reqno]{amsart}
\usepackage{amssymb, hyperref}
\usepackage{amsthm}
\usepackage{graphicx}
\usepackage[all,2cell]{xy}
\usepackage{verbatim}
\usepackage{tikz}
\usepackage{tikz-cd}
\usepackage{cite}
\usepackage{hyperref}
\usetikzlibrary{matrix,arrows,decorations.pathmorphing}
\usepackage{mathrsfs}
\numberwithin{equation}{section}
\usepackage{amssymb, amscd, hyperref}
\usepackage{amsthm}
\usepackage{graphicx}
\usepackage[all,2cell]{xy}
\usepackage{verbatim}
\usepackage{stmaryrd}
\usepackage{tikz}
\usepackage{ytableau}
\usepackage{xypic}
\usepackage{tikz-cd}
\usepackage{hyperref}
\usepackage{calc}
\usepackage{lineno}
\usepackage[utf8]{inputenc}
\usetikzlibrary{decorations.markings}
\tikzstyle{vertex}=[circle, draw, inner sep=0pt, minimum size=6pt]
\newcommand{\vertex}{\node[vertex]}
\usetikzlibrary{matrix,arrows,decorations.pathmorphing}
\usepackage{mathrsfs}
\usetikzlibrary{calc, through}
\numberwithin{equation}{section}
\newtheorem*{theorem*}{Theorem}
\newtheorem*{corollary*}{\bf Corollary}
\newtheorem*{remark*}{\bf Remark}
\usepackage{lipsum}
\usetikzlibrary{matrix,arrows,decorations.pathmorphing}
\usepackage{lineno}
\newtheorem{theorem}{Theorem}[section]
\newtheorem{corollary}[theorem]{Corollary}
\newtheorem{definition}[theorem]{Definition}

\newtheorem{lemma}[theorem]{Lemma}

\newtheorem{proposition}[theorem]{Proposition}
\newtheorem{remark}[theorem]{Remark}
\usepackage{csquotes}
\usepackage{epigraph}

\title[Torus quotients of Schubert varieties]
{On the Torus quotients of Schubert varieties}

 \author[N.~Chary Bonala]{Narasimha Chary Bonala}
\address{ Narasimha Chary Bonala\\
Max Planck Institute for Mathematics\\
Vivatsgasse 7, Bonn\\
Germany
}
\curraddr{
Fakult\"{a}t f\"{u}r Mathematik\\
Ruhr-Universit\"{a}t Bochum, D-44780 Bochum\\
Germany\\
Email: Narasimha.bonala@rub.de
}

\author[S.~K.~Pattanayak]{Santosha Kumar Pattanayak}
\address{
Santosha Kumar Pattanayak\\
Department of Mathematics and Statistics, IIT Kanpur\\
Kanpur, India\\
Email:santosha@iitk.ac.in\\
}

\begin{document}
\maketitle
\begin{linenomath*}

\begin{abstract}
 In this paper, we consider the GIT quotients of Schubert varieties for the action of a maximal torus. 
 We describe the minuscule Schubert varieties for which the 
 semistable locus is contained in the smooth locus. As a consequence, we study the smoothness of torus quotients of Schubert varieties in the Grassmannian.  We also prove that the torus quotient of any Schubert variety in the homogeneous space $SL(n, \mathbb C)/P$ is projectively normal with respect to the line bundle $\mathcal L_{\alpha_0}$ and the quotient space is a projective space, 
where the line bundle $\mathcal L_{\alpha_0}$ and the parabolic subgroup $P$ of $SL (n, \mathbb C)$ are associated to the highest root $\alpha_0$.
\end{abstract}

\end{linenomath*}
%%%%%%%%%%%%

%\tableofcontents
\section{Introduction}

The geometry (both symplectic and algebraic) of the quotients 
of projective homogeneous spaces 
by a maximal torus 
has been extensively studied in recent years; Allen Knutson calls 
them as {\it weight varieties} in his thesis \cite{Knu}. 
The construction of such quotients involves the choice of a 
linearized line bundle.
The dependence of the geometry of the quotient on the choice of a linearization was studied by Hu in \cite{Hu}, and in a more general setting by Dolgachev 
and Hu in \cite{DoHu}. The cohomology spaces of nonsingular weight varieties for $SL(n, \mathbb C)$ were computed by Goldin \cite{Goldin}. Special cases of 
weight varieties have been studied since the nineteenth century; for example a GIT quotient $(\mathbb {CP}^{ k-1})^n \sslash PGL(k, \mathbb C)$ is isomorphic to a GIT quotient
$Gr_k(\mathbb C^n)\sslash T$ by the Gelfand-MacPherson correspondence (here $Gr_k(\mathbb C^n)$ denotes the Grassmannian and $T$ denotes a maximal torus). The symplectic quotient of  Schubert varieties by a $1$-dimensional torus is studied in \cite{mare}. In this paper, we study the GIT quotients of Schubert varieties by a maximal torus.

To describe our results explicitly, we set up some notation. Let $G$ be a semisimple simply connected complex algebraic group. 
Let $T$ be a maximal torus of $G$ and $B$ be a Borel subgroup of $G$ containing $T$.
Let $P$ be a parabolic subgroup of $G$ containing $B$. In \cite{CSS, kannan1998torus, kannan1999torus, strickland2000quotients, kannan20091torus} 
the authors consider the problem of studying the torus quotients of flag varieties $G/P$. The papers \cite{pattanayak2014minimal, Kannansantosh, kannan2018torus} address the existence of semistable points in Schubert varieties and Richardson varieties for the action of a maximal torus.
In this paper, we study the smoothness of torus quotients of Schubert varieties in $G/P$ when $G$ is of type $A$ and $P$ is a maximal parabolic subgroup containing $B$.
When $G$ is of type other than $A$, we describe the minuscule Schubert varieties such that the semistable locus is contained in the smooth locus. To describe our results, we use a combinatorial description of quivers associated with minuscule Schubert varieties, which are introduced in \cite{perrin2007small}. They
generalize the notion of Young diagrams to other minuscule Grassmannians. Following \cite{perrin2007small}, we define the holes and the essential holes of a quiver and we use them to describe our results. For more details see Sec. \ref{quivers} or \cite{perrin2007small}.

 Let $\mathcal {L}_{\omega}$ be the homogeneous line bundle on the minuscule Grassmannian $G/P$ associated to a minuscule fundamental weight $\omega$. 
Let $W^P$ be the Weyl group associated to $P$. Given a $w\in W^{P}$, we denote by $X_{P}(w)_T^{ss}(\mathcal L_{\omega})$, the set of semistable points in the Schubert variety $X_{P}(w)$ with respect to $\mathcal L_{\omega}$ for the action of $T$. 
 
Recall that for minuscule Grassmannians, there is a unique minimal Schubert variety $X_P(v)$  admitting semistable points (i.e., $X_{P}(v)_T^{ss}(\mathcal L_{\omega})\neq \emptyset$) (see \cite[Lemma 1.7]{kannan20091torus}). In the sequel, we fix such a $v\in W^P$.  Let $Q_w$ be the quiver variety associated to $w\in W^P$. 
 
\begin{theorem*}[see Theorem \ref{quiver1}]
  The semistable locus  is contained in the smooth locus for $X_P(w)$ if and only if the quiver $Q_{v}$ contains all the essential holes of $Q_{w}$.
  \end{theorem*}
  
    As a consequence, we classify the Schubert varieties in the Grassmannian $Gr(r, n)$ for which the torus quotient is smooth when $r$ and $n$ are coprime.  We also give a direct proof using the Young diagrams (see Proposition \ref{semi-singular}). In terms of quivers Corollary \ref{smoothquotient} reads as follows:
\begin{corollary*}
 Let $G=SL(n, \mathbb C)$ and let $1< r < n-1$ be such that $gcd(n,r)=1$. Let $P_r$ be the maximal parabolic subgroup of $G$ associated to the simple root $\alpha_r$. Assume that $X_{P_r}(w)_T^{ss}(\mathcal L_{\omega_r})$ is nonempty, then the quotient $X_{P_r}(w)_T^{ss}(\mathcal L_{\omega_r})\sslash T$ 
 is smooth if the quiver $\mathcal Q_v$ contains all the essential holes of $Q_{w}$.
\end{corollary*}

Let $P$ be the parabolic subgroup of $SL(n, \mathbb C)$ corresponding to $\alpha_0=\omega_1+\omega_{n-1}$ (the highest root of $SL(n,\mathbb C)$) and let $\mathcal L_{\alpha_0}$ be the homogeneous line bundle on $G/P$ associated to $\alpha_0$. Fix a maximal torus $T\subseteq B\subseteq P$. Recall that since $\alpha_0$ is in the root lattice, by Kumar's result in \cite{Kumar}, the line bundle $\mathcal L_{\alpha_0}$ descends to the torus quotient.  

\begin{theorem*}[see Theorem \ref{projectivenormality} and Corollary \ref{quotient}]\
Assume that  $X_{P}(w)_T^{ss}(\mathcal L_{\alpha_0})$ is nonempty. Let $\mathcal M$ be the descent of $\mathcal L_{\alpha_0}$ to the quotient $X_P(w)^{ss}_{T}(\mathcal L_{\alpha_0})\sslash T$. Then
\begin{enumerate}
 \item
 The polarized variety 
$(X_P(w)^{ss}_{T}(\mathcal L_{\alpha_0})\sslash T, \mathcal M)$ is projectively normal.
\item 
 The quotient $X_P(w)^{ss}_T(\mathcal L_{\alpha_0})\sslash T$ is isomorphic to a projective space.
\end{enumerate}
 \end{theorem*}
 
{\it When $P$ is as in the above theorem, i.e., $P=P_1\cap P_{n-1} \subset SL(n, \mathbb C)$, then the homogeneous space $SL(n, \mathbb C)/P$ can be viewed as the incidence variety between lines and hyperplanes in $\mathbb C^{n}$. Hence, one would expect to recover the projective spaces as quotients of $G/P$ for the torus action. The above theorem confirms that, it is indeed true for the Schubert varieties in $G/P$.}

\noindent We prove this theorem by applying the theory of standard monomials on Schubert varieties. 

The layout of the paper is as follows: Section 2 consists of some preliminaries and we establish a relation between the semistable locus and the singular locus of Schubert varieties. In Sec. 3, we deal with Schubert varieties in the Grassmannians 
and in Sec. 4, we recall the notion of quivers associated with minuscule Schubert varieties and we prove the above stated results.
In Sec. 5, we consider the torus quotients of Schubert varieties for $SL(n, \mathbb C)$ in the case of the highest root.
\section{Preliminaries}\label{Section 2}

In this section, we briefly recall the preliminaries from algebraic groups, Lie algebras and geometric invariant theory.
%to set up notation. 
We refer to \cite{humphreys2012introduction, humphreys2012linear, GITmumford}
for a more detailed exposition.
Let $G$ be a semisimple simply connected complex algebraic group. 
Let $T$ be a maximal torus of $G$, $B$ a Borel subgroup of $G$ containing 
$T$ and let $U$ be the unipotent radical of $B$. 
Let $N_G(T)$ be the normalizer of $T$ in $G$ and $W=N_G(T)/T$ be the Weyl group of 
$G$ with respect to $T$. We  denote by $\Phi$, the set of roots with respect to 
$T$ and by $\Phi^{+}$, the set of positive roots with respect to $B$.
For the enumeration of roots we refer to \cite{humphreys2012introduction}. 
Let $S=\{\alpha_1, \ldots, \alpha_l\}\subseteq \Phi^{+}$ 
be the set of simple roots. For a given root $\alpha$, we denote $U_{\alpha}$ by the one-dimensional $T$-stable subgroup of $G$.
Set $$R(w):=\{\alpha\in \Phi^+:w(\alpha)<0\}~\text{and}~U_w:=\prod_{\alpha \in R(w)}U_{\alpha}$$ (see \cite[p.144]{Springer}).
Given a subset $I\subseteq S$, we denote 
$$W^I =\{w\in W| w(\alpha)>0, \, \alpha\in I\}$$ and $W_I$ be the subgroup of $W$
generated by the simple reflections $s_\alpha, \alpha \in I$. Then, every element 
$w \in W$ is uniquely expressed as the product $w=w^I.w_I$ with $w^I \in W^I$ and 
$w_I \in W_I$. 
We denote $w_0$ by 
the longest element of $W$ with respect to $S$.
Consider the set of characters $X(T)$ (respectively, one-parameter subgroups $Y(T)$) of $T$. Consider the canonical non-degenerate bilinear form $$\langle . ,  .\rangle : E_1\times E_2 
\longrightarrow \mathbb R,$$ where $$E_1:= X(T)\otimes \mathbb R~\text{and}~ E_2=Y(T)\otimes \mathbb R.$$ 
We have 
$s_{\alpha}(\chi)=\chi-\langle \chi, \check{\alpha}\rangle \alpha$ for all 
$\alpha \in \Phi$ and $\chi \in E_1$, where $\check \alpha$ is the coroot of $\alpha$. Set $s_i=s_{\alpha_i} \,\ \forall 
\,\ i=1,2, \ldots, l$ and denote by $\{\omega_i: i=1,2,\ldots, l\} \subset E_1$, the set of  
fundamental weights; i.e., $\langle \omega_i, \check{\alpha_j} \rangle = 
\delta_{ij}$ for all $i,j = 1,2, \ldots, l$. 

 Let $P$ be the parabolic subgroup corresponding to a subset $I$ of $S$.  We also denote $W/W_{I}=W^I$ by
$W^P$ when $P$ corresponds to $I\subset S$.  
For the left action of $T$ on $G/P$, there are only finitely many fixed points in $G/P$ and are given by  
$$\{e_w := wW_I: w \in W/W_I\}.$$ For a given $w \in W/W_I$, the $B$-orbit $C_P(w) := 
Be_w= BwP/P$ in $G/P$ is a locally closed subset of $G/P$, called the {\it 
Schubert cell}. The Zariski closure of $C_P(w)$ with the canonical reduced 
structure is the {\it Schubert variety} associated to $w$, and is denoted by 
$X_P(w)$. Thus the Schubert varieties in $G/P$ are indexed by $W^I$.
Note 
that if $P = B$, then $W_I = \{id\}$, and the Schubert varieties in $G/B$ are 
indexed by the elements of $W$. We denote by $X(w)$ the Schubert variety corresponding 
to $w \in W$. There is a cell decomposition of Schubert varieties 
$X_P(w)$ in $G/P$ that is $$X_P(w)=\amalg_{v\leq w}BvP/P$$ (Bruhat-decomposition). Here, $\leq$ denotes the Bruhat order on $ W^P$. 

We now recall the definition of the Hilbert-Mumford numerical function and the
definition of semistable points from \cite{GITmumford}. We also refer to \cite{newstead} 
for notation in geometric invariant theory. 

Let $X$ be a projective variety with an action of %a reductive group
$G$. A 
point $x \in X$ is said to be semistable with respect to a $G$-linearized very ample
line bundle $\mathcal L$ if there is a positive integer $m \in \mathbb N$, and 
a $G$-invariant section $s \in H^0(X, \mathcal L^m)$ with $s(x) \neq 0$.
Let $\lambda$ be a one-parameter subgroup of $G$, i.e, $\lambda:\mathbb G_m\to G$ is a algebraic group homomorphism.
 For a given point $x \in \mathbb 
P(H^0(X,\mathcal L)^*)$, we write $\hat{x}\in H^0(X,\mathcal L)^*$ in the form $$\hat{x}= \sum_{i=1}^rv_i,$$ where $v_i$ is a
weight vector of $\lambda$ with weight $m_i$. 
Then the Hilbert-Mumford numerical function is defined by 
\begin{linenomath*}\[\mu^{\mathcal L}(x, 
\lambda):=-min\{m_i: i =1, \ldots, r\}.\]\end{linenomath*} The  Hilbert-Mumford criterion says 
that the point $x$ is semistable (respectively, stable) if and only if $\mu^{\mathcal L}(x, \lambda) \geq 0$ (respectively, $\mu^{\mathcal L}(x, \lambda) > 0$ ) for all 
one-parameter subgroups $\lambda$.

For any character $\chi$ of $P$, we denote by $\mathcal L_{\chi}$, the homogeneous line 
bundle on $G/P$ associated with $\chi$. We denote by $X_P(w)_T^{ss}
(\mathcal L_{\chi})$ the semistable points of $X_P(w)$ for the action of $T$ 
with respect to the line bundle $\mathcal L_{\chi}$. 
For more details on semistable points in the case of flag varieties we refer to \cite{seshadri1972quotient}.

From now on assume that $\chi$ is a dominant character of $P$ and we set $\mathcal L:=\mathcal L_{\chi}$. We define, 
\begin{align*}
\mathcal E_{w}^{ss}:=\{ v\in W^P: v\leq w ~\mbox{and}~ v~\mbox{is minimal with }~X_P(v)_T^{ss}(\mathcal L)\neq \emptyset\}~\text{and}~
\end{align*}  
\begin{align*}
\mathcal E_w^{sing}:=\{v\in W^P: v\leq w ~\mbox{and}~ v~\mbox{is maximal such that}~ \\ e_v~\mbox{is a singular point in}~X_P(w) \}.
\end{align*}

\noindent We first recall a useful observation (see \cite[Chapter 4.4]{billey2000singular} for more details).
\begin{remark}\label{1}\
\begin{enumerate}
    \item The point $e_w\in X_P(w)$ is always a smooth point in $X_P(w)$; and 
    \item If the point $e_{id}$ is smooth in $X_P(w)$, then $X_P(w)$ is smooth.
\end{enumerate}
\end{remark}

\noindent The following lemma characterizes the Schubert varieties for which the semistable locus is contained in the smooth locus.   

\begin{lemma} \label{3.4} The set 
 $X_P(w)^{ss}_T(\mathcal L)$ is contained in the smooth locus of $X_P(w)$ if and only if $w_1\ngeqslant w_2 ~~\mbox{for all} ~w_1\in \mathcal E_w^{sing} ~\mbox{and} ~ w_2\in \mathcal E_{w}^{ss}$. 
\end{lemma}
\begin{proof}
 Assume that $X_P(w)^{ss}_T(\mathcal L)$ is contained in the smooth locus of $X_P(w)$. If $X_P(w)$ is smooth or $X_P(w)^{ss}_T(\mathcal L)=\emptyset$, then there is nothing to prove. 
 So assume that $X_P(w)$ is singular and $X_P(w)^{ss}_T(\mathcal L)\neq \emptyset$.
Then, there exists $v\leq w$ minimal such that $X_P(v)_T^{ss}(\mathcal L)\neq \emptyset$. Let $w'\in \mathcal E_w^{sing}$.
If $v\leq w' \leq w$, then we see that \begin{linenomath*} $$X_P(v)\subset X_P(w')\subset Sing(X_P(w))~\text{and so} ~X_P(v)^{ss}_T(\mathcal L)\subset Sing(X_P(w)).$$ \end{linenomath*} This is a contradiction to the assumption. 

Otherwise, we have either $w'<v$ or $w'$ and $v$ are not comparable. By the minimality of $v$, it follows that $$X_P(w')\cap X_P(v)^{ss}_T(\mathcal L)=\emptyset.$$ Thus, we conclude that $w_1\ngeqslant w_2 ~~\mbox{for all} ~w_1\in \mathcal E_w^{sing} ~\mbox{and} ~ w_2\in \mathcal E_{w}^{ss}$.

Now assume that $w_1\ngeqslant w_2 ~~\mbox{for all} ~w_1\in \mathcal E_w^{sing} ~\mbox{and} ~ w_2\in \mathcal E_{w}^{ss}$. We
suppose there exists $x\in X_P(w)_T^{ss}(\mathcal L)$ such that $x$ is singular in $X_P(w)$. 
Note that  $$X_P(w)_T^{ss}(\mathcal L)\subseteq \coprod _v BvP/P$$ where $w\geq v\geq v'$ for some $v'\in \mathcal E_w^{ss}$.
It is clear that $x\in BvP/P$ for some such $v$. Since $b\cdot x$ is also singular for all $b\in B$, it follows that $$BvP/P \subset Sing(X_P(w))~\text{and hence,}~ X_P(v)\subset Sing(X_P(w)).$$
Thus, there exists $v'\leq v$ such that $e_{v'}$ is singular in $X_P(w)$ and $v'\in \mathcal E_w^{ss}$.
This implies that $v'\in \mathcal E_w^{ss}\cap \mathcal E_w^{sing}$, a contradiction.
Hence, $X_P(w)^{ss}_T(\mathcal L)$ is contained in the smooth locus of $X_P(w)$.
\end{proof}

\section{Schubert varieties in the Grassmannian}
Let $G=SL(n,\mathbb C)$ and let $P_r$ be the maximal parabolic subgroup corresponding to the simple root $\alpha_r$. Note that $Gr(1,n)\cong Gr(n-1,n) \cong \mathbb P^{n-1}$. So in these two cases the Schubert varieties and their torus quotients are all smooth. So from now on we assume that $r\in \{2,3,\ldots, n-2\}$. We denote $\mathcal L_r$ by the line bundle 
on $G/P_r=Gr(r,n)$ associated to the fundamental weight $\omega_r$. In this section, we give a criterion for Schubert varieties in the Grassmannian $Gr(r,n)$
for which the semistable locus for the action of a maximal torus is contained in the smooth locus. When $r$ and $n$ are coprime to each other we also show 
that the torus quotient is smooth. In the next section, we prove a general result for minuscule Schubert varieties using quivers associated with them. 

We start by recalling the correspondence between the elements of $W^{P_r}$ and the Young diagrams. We follow the convection in \cite[Chapter 3]{manivel2001symmetric}.  

 The permutations giving rise to distinct Schubert varieties in $Gr(r,n)$ are of the form $w=(w_1,\ldots,w_n)$ 
where  $w_{1}<\cdots<w_k$, $w_{r+1}<\cdots<w_n$. Therefore, it is enough to record the permutations of $S_n$ as $w=(w_1,\ldots,w_r)$, which will be called a {\em Grassmannian permutation}. Thus, we can identify $W^{P_r}$ with the set $$I(r, n):=\{(i_1, i_2, \ldots, i_r): 1\leq i_1<i_2< \cdots <i_r\leq n\}.$$ 

For a given $(i_1, i_2, \ldots, i_r) \in I(r, n)$, we associate a partition $\lambda$ contained in an $r\times n-r$ rectangle, that is a decreasing sequence of integers $$n-r\geq \lambda_1\geq \ldots \geq \lambda_{r}\geq 0.$$ The association is given by $\lambda_j=n-r+j-i_j$ for all $j$.

The Young diagram of the partition $\lambda$ is obtained by lining up cells/boxes, from top to bottom, rows of length given by parts of $\lambda$ and sharing the same left most column. Note that there are $\lambda_{i}$ cells/boxes in $i`$th row.

We say a cell/box at the position $(i, j)$ is a corner of the  Young diagram $\lambda$ if $\lambda_i\neq n-r$ and we obtain a valid Young diagram after removing  it.

Given two partitions $\lambda$ and $\mu$ we say $\mu \subset \lambda$ if the Young diagram for $\mu$ fits inside the Young diagram for $\lambda$. We can also parametrise Schubert varieties in the Grassmannian $Gr(r, n)$ by the partitions in the $r\times n-r$ rectangle. We denote by $X_{\lambda}$ the Schubert variety corresponds to $\lambda$.

\noindent We recall the following result (see \cite[Proposition 3.2.3 ]{manivel2001symmetric}):
\begin{lemma}\label{contain} For given partitions  $\lambda$ and $\mu$ inside the $r\times n-r$ rectangle, we have $$X_{\lambda}\subset X_{\mu} ~~\text{if and only if}~~ \mu\subset \lambda.$$
\end{lemma}

The following theorem describes the singular locus of Schubert varieties in Grassmannians (see \cite[Theorem 3.4.4]{manivel2001symmetric}).  

\begin{theorem}\label{singularlocus} Let $X_{\mu}$ be a Schubert variety in $Gr(r, n)$. Let $T(\mu)$ be the set of all partitions obtained from $\mu$ as follows: if $(i,j)$ is a corner of $\mu$, we add
 to the diagram of $\mu$ the cell $(i+1,j+1)$, and complete a partition from it in a minimal way. Then,  
 \begin{equation}\label{singularlocus1}
 Sing(X_{\mu})=\cup _{\mu'\in T(\mu)}X_{\mu'}
 \end{equation} is the decomposition of singular locus of the Schubert variety $X_{\mu}$ into irreducible components.
\end{theorem}

\begin{remark}\label{remarksmooth} The smooth Schubert varieties $X_{\lambda}$ are those which correspond to partitions $\lambda$ which have diagrams complementary in the $r\times n-r$ rectangle, to smaller rectangles. Further, smooth Schubert varieties are sub-Grassmannians (see \cite[Remark 3.4.5, p.117]{manivel2001symmetric}).   

\end{remark}

Recall the notation of $\mathcal E^{ss}_w$ from Sec. \ref{Section 2}. 
Inside the Grassmannian $Gr(r, n)$ there is a unique minimal Schubert variety  $X_{P_r}(v)$  such that $X_{P_r}(v)^{ss}_T(\mathcal L_r)\neq \emptyset$ thanks to \cite[Lemma 1.7 and 2.7]{kannan20091torus}. Thus, it follows that the set $\mathcal E^{ss}_w$ is either empty or $\{v\}$ for any $w\in W^{P_r}$. 
 
\begin{proposition}\label{semi-singular}Let $w=(w_1, \ldots, w_r)\in W^{P_r}$ be such that $X_{P_r}(w)^{ss}_T(\mathcal L_r)$ is nonempty. Then, the semistable locus $X_{P_r}(w)^{ss}_T(\mathcal L_r)$ is contained in the smooth locus of $X_{P_r}(w)$ if 
and only if $w_i<v_{i+1}$ whenever $w_{i+1}>w_{i}+1$. 
\end{proposition}

\begin{proof} 
   Let $\mu=(n-r\geq \mu_1\geq \ldots \geq \mu_r \geq 0 )$ (respectively, $(n-r\geq\lambda_1\geq \cdots \geq \lambda_r\geq 0)$) be the partition corresponding to $w\in W^{P_r}$ (respectively, $v$). 
 Since $X_{\lambda}$ is contained in $X_{\mu}$, the Young diagram corresponding to $\mu$ is contained in the Young diagram corresponding to $\lambda$ (see Lemma \ref{contain}). 
Recall the definition of $T(\mu)$ from Theorem \ref{singularlocus}.
 By Lemma \ref{3.4} and Theorem \ref{singularlocus},  the semistable locus $(X_{\mu})^{ss}_T(\mathcal L_r)$ is contained in the smooth locus of $X_{\mu}$ if and only if  there is no $\mu'\in T(\mu)$ containing $\lambda$.
 
 By the description of $T(\mu)$ and the minimality of $\lambda$, there is no $\mu'\in T(\mu)$ containing $\lambda$ if and only if $\lambda_{i+1}<\mu_i+1$ whenever $(i, j)$ is a corner of $\mu$.  Precisely, by the definition of a corner in the Young diagram of $
 \mu$,  if $\mu_{i+1}<\mu_i<n-r$ then $\lambda_{i+1}<\mu_i+1$.  Equivalently, if $w_{i+1}>w_i+1$, then $v_{i+1}>w_i$. This finishes the proof.
 \end{proof}
 
{\it In an earlier version of the paper, there was an error in our calculations that results in changes in the conditions on $w_i$'s in the statement of Proposition \ref{semi-singular}. We thank Sarjick Bakshi, who pointed out this and for informing us about his statement in ~\cite{Sarjick} which is under review.}

\begin{corollary}\label{smoothquotient}Let $gcd (r, n)=1$.  If $w\in W^{P_r}$ satisfies the condition in Proposition \ref{semi-singular},   %$\mu$ is either smooth or as above, 
  then the torus quotient $X_{P_r}(w)^{ss}_T(\mathcal L_r)\sslash T$ is smooth.
\end{corollary}
\begin{proof} We have that the semistable locus  $X_{P_r}(w)^{ss}_T(\mathcal L_r)$ is contained in the smooth locus of $X_{P_r}(w)$ thanks to Proposition \ref{semi-singular}.
Since $gcd(r, n)=1$, by \cite[Theorem 3.3]{kannan1998torus}, the following equality holds: \[(G/P_r)^{ss}_T(\mathcal L_r)=(G/P_r)^s_T(\mathcal L_r).\]
 Since 
 \[X_{P_r}(w)^{ss}_T(\mathcal L_r)=X_{P_r}(w)\cap (G/P)_T^{ss}(\mathcal L_r) ~~\mbox{and}~~\]
 \[X_{P_r}(w)^{s}_T(\mathcal L_r)=X_{P_r}(w)\cap (G/P)_T^{s}(\mathcal L_r),\]
   it follows that  
 $X_{P_r}(w)^{ss}_{T}(\mathcal L_r)=X_{P_r}(w)^s_T(\mathcal L_r)$. 
Note that for any $x \in X_{P_r}(w)^{ss}_T(\mathcal L_r)$, the $T$ orbit of $x$ is closed 
 in $X_{P_r}(w)^{ss}_T(\mathcal L_r)$, the stabilizer of $x$ is finite and is contained in the center (see \cite[p.194]{kannan2014git} for more details). Therefore, by working with the adjoint 
 group of $SL_n(\mathbb C)$, we may assume that the isotropy subgroup of $x$ is trivial. In such a case, by using the above arguments we note that the quotient morphism 
 \begin{linenomath*}
   $$X_{P_r}(w)^{ss}_T(\mathcal L_r)\longrightarrow  
 X_{P_r}(w)^{ss}_T(\mathcal L_r)\sslash T$$ 
 \end{linenomath*}
 is a principal $T$-bundle. Hence, we conclude the proof.
\end{proof}

\begin{remark}
The converse of Corollary \ref{smoothquotient} is also true. Precisely, if the torus quotient $X_{P_r}(w)^{ss}_T(\mathcal L_r)\sslash T$ is smooth then $w\in W^{P_r}$ satisfies the condition in Proposition \ref{semi-singular}.
From the proof it easily follows that the quotient $X_{P_r}(w)^{ss}_T(\mathcal L_r)\sslash T$ is an orbit space. Hence, a point in $X_{P_r}(w)^{ss}_T(\mathcal L_r)$ is smooth if and only if its image is smooth in the quotient $X_{P_r}(w)^{ss}_T(\mathcal L_r)\sslash T$. 
This is observed recently in \cite{Sarjick} when this article is in the publishing process. 
\end{remark}

\noindent We conclude this section by the following observation. Recall that when $1<r<n-1$, the Grassmannian $G/P_{r}$ is not a projective space. 
\begin{lemma}\label{min} 
 Assume that $w\in W^{P_r}$ is minimal with $X_{P_r}(w)_T^{ss}(\mathcal L_r)\neq \emptyset$.
 Then the Schubert variety $X_{P_r}(w)$ is singular.
\end{lemma}
\begin{proof}
  We know by \cite[Lemma 2.7]{kannan20091torus} that any minimal element $w \in W^{P_r}$ satisfying $X_{P_r}(w)_T^{ss}(\mathcal L_r)\neq \emptyset$ is of the form
 $$w=(s_{a_1}\cdots s_1)\cdots(s_{a_r}\cdots s_r),$$ where $a_i$'s are defined as follows:
 Let $n=qr+t$ with $1\leq t \leq r$; then
 \begin{linenomath*}
   \[a_i=\begin{cases}
       i(q+1) &~\mbox{if}~ i\leq t-1\\
       iq+(t-1) &~\mbox{if}~ t\leq i\leq r.
      \end{cases}
\]
\end{linenomath*}
 
 It is easy to see that  $a_1<a_2 < \ldots <a_r$. Then the Grassmannian 
 permutation $w$ of the form $$w=(a_1+1,a_2+1,a_3+1, \ldots ,a_r+1).$$ 
 For such $w\in W^{P_r}$, the associated partition is given by $$(\lambda_1\geq \lambda_2\geq \cdots \lambda_r\geq 0) ~\text{with}~ \lambda_i=n-r+i-a_i-1.$$ Since $2\leq r\leq n-2$, by the definition of $a_i$'s, it is clear that not all $\lambda_i$'s are equal. Then, there exist  $i$ and $j$ such that $\lambda_i\neq \lambda_j$. 
 Therefore, the complement of the Young diagram corresponding to $(\lambda_1\geq \lambda_2\geq \cdots \lambda_r\geq 0)$ is not a rectangle (inside $r\times(n-r)$ rectangle).
 Thus, it follows that the Schubert variety $X_{P_r}(w)$ is not smooth, thanks to Remark \ref{remarksmooth}.  
 \end{proof}

 \section{Minuscule Schubert varieties and quivers}
 
\subsection{}\label{quivers}
In this section, we classify minuscule Schubert varieties for which the semistable locus is contained in the smooth locus. 
First, we recall the definition of minuscule Schubert varieties from \cite[Definition 2.1]{lakshmibai1978geometry} and quivers associated with them from \cite{PERRIN2009505}.
\begin{definition}
 A fundamental weight $\omega$ is called {\it minuscule} if it satisfies one of the following equivalent conditions:
 \begin{enumerate}
  \item $\langle \omega,\check{\beta} \rangle \leq 1$  for all $\beta\in \Phi^+$.
  \item Every weight of $V_{\omega}$ is in the orbit of $\omega$. 
 \end{enumerate}

\end{definition}

Let $\omega$ be a minuscule weight and let $P_{\omega}$ be the associated parabolic subgroup. 
The partial flag variety $G/P_{\omega}$ is called {\it minuscule} and the Schubert varieties in $G/P_{\omega}$ are called {\it minuscule Schubert varieties}.
Table 1 describes the minuscule weights:\\

\begin{table}[h]
 \begin{center}
\begin{tabular}{ |c|c| } 
 \hline
 $G$ & Minuscule weight\\
 \hline
  Type-A & All fundamental weights. \\ 
 \hline
 Type-B & $\omega_{n}$\\
\hline
Type-C & $\omega_1$\\
\hline
Type-D & $\omega_1, \omega_{n-1}$ and  $\omega_{n}$\\
\hline
Type-$E_6$ & $\omega_1~\mbox{and}~ \omega_6$\\
\hline
Type-$E_7$ & $\omega_7$\\ 
\hline
Type-$E_8, F_4, G_2$ & None\\
 \hline
\end{tabular}
\end{center}
\caption{Minuscule weights}
\end{table}

Let $\widetilde w=s_{\beta_1}\cdots s_{\beta_r}$ be a reduced expression of  $w\in W^{P_{\omega}}$ in terms of simple reflections. 
We recall the definition of quiver $Q_{\widetilde w}$ associate with $\widetilde w$.

\begin{definition}\

 \begin{enumerate}
  \item The successor $s(i)$ and the predecessor $p(i)$ of an element $i\in \{1,2,\ldots, r\}$ are the elements defined as 
  $$s(i):=min\{j\in\{1,2,\ldots, r\}: j>i ~\mbox{and}~ \beta_j=\beta_i\},~\text{and}~$$ $$p(i):=max\{j\in\{1,2,\ldots, r\}: j<i ~\mbox{and}~ \beta_j=\beta_i\}.$$
  \item The quiver $Q_{\widetilde w}$ is the quiver whose vertices is the set $\{ 1, 2, \ldots, r\}$ and whose arrows are given by the following way: 
  there is an arrow from $i$ to $j$ if $\langle \beta_{i},\check{\beta_j} \rangle \neq 0$ and   $i<j<s(i)$ (or $i<j$ if there is no $s(i)$).
 \end{enumerate}
\end{definition}
\begin{remark}[see \cite{PERRIN2009505}]\

\begin{enumerate}
 \item 
 The quiver $Q_{\widetilde w}$ does not depend on the reduced expression we choose for $w\in W^{P_{\omega}}$ and we denote it by $Q_{w}$. 
\item Let $Q_{w_0}$ denote the quiver obtained from the longest element in $W^P$. Then the quiver $Q_{w}$ is a subquiver of $Q_{w_0}$.
\end{enumerate}

 \end{remark}

\noindent There is a natural order on the quiver $Q_{w}$ given by $i\leq j$ if there is an oriented path from $j$ to $i$.
The quiver $Q_{w}$ encodes some of the geometric properties of the Schubert variety $X_P(w)$.
We recall the following notions:

\begin{enumerate}
 \item 
A {\it peak} to be any vertex of $Q_{w}$ that is maximal for the partial order $\leq$. 
%Denote by $Peaks(Q_{w})$ the set of all peaks of $Q_{w}$.

\item 
We define a {\it hole} of the quiver $Q_{w}$ to be any vertex $i$  satisfying one of the following properties:
\begin{enumerate}
    \item [(i)] the vertex $i$ is in $Q_{w}$ but $p(i)\notin Q_{w}$, and there are exactly two vertices $j_1\geq i$ and 
$j_2\geq i$ in $Q_{w}$ with $\langle \beta_j, \beta_i \rangle\neq 0$ for $j=j_1, j_2$;
\item[(ii)] the vertex $i$ is not  in $Q_{w}$, $s(i)$ does not exist in $Q_{w_0}$. 

\end{enumerate}
\end{enumerate}
Because the vertex of the second type of holes is not a vertex in $Q_{w}$, we call such a hole a {\it virtual hole} of $Q_{w}$. 
%We denote with $Holes(Q_{w})$ the set of holes of $Q_{w}$.

\noindent Now we recall the following result from \cite[Proposition 1.11]{PERRIN2009505}:
\begin{proposition}\label{smoothnesscriterion}
 A Schubert variety $X_P(w)$ is smooth if and only if the holes of its quiver $Q_{w}$ are virtual.
\end{proposition}

To describe the irreducible components of the singular locus of $X_P(w)$ in terms of the quiver, we need the following definition:

{\bf Definition:}\

\begin{enumerate}
 \item 
 Let $i$ be a vertex of $Q_{w}$ and define the subquiver $Q_{w}^i$ of $Q_{w}$ as the full subquiver 
  containing the set of vertices $\{j\in Q_{w}: j\geq i\}$. We denote by $Q_{w, i}$ the full subquiver of  $Q_{w}$
containing the vertices of $Q_{w}\setminus Q_{w}^i$.

  \item A hole $i$ of the quiver $Q_{w}$ is said to be {\it essential} if it is not virtual and there is no hole in the subquiver $Q_{w}^i$. 
  
\item For a vertex $i$ of $Q_{w}$, we denote by $w^i$(respectively, $w_i$) the elements in $W^{P_{\omega}}$ such that $Q_{w^i}=Q_{w}^i$ 
 (respectively, $Q_{w_i}=Q_{w, i}$). 
 
\end{enumerate}

 \begin{theorem}[see Theorem 2.7 of \cite{PERRIN2009505}]\
 
 \begin{enumerate}
  \item 
   Let $i$ be an essential hole of $Q_{w}$; then the Schubert subvariety $X_P(w_{i})$ of $X_P(w)$ whose quiver is $Q_{w, i}$, is an 
  irreducible component of the singular locus of $X_P(w)$.
  
  \item All the irreducible components of the singular locus are obtained in this way.
  
 \end{enumerate}

\end{theorem}

\begin{remark}\label{exclude}
 It is known that the flag varieties in type $A$ corresponding to $\omega_r$, $r=1$, $n-1$ and in type $C$ corresponding to $\omega_1$
are projective spaces. Hence, we exclude these three cases.
\end{remark}

Let $(G, \omega)$ be a pair with a group $G$ and a minuscule weight $\omega$. Let $\mathcal L_{\omega}$ be the 
corresponding line bundle on $G/P_{\omega}$. 
For minuscule Grassmannians $G/P_{\omega}$, there is a unique minimal Schubert variety $X_{P_{\omega}}(v)$ such that $X_{P_{\omega}}(v)^{ss}_T(\mathcal L_{\omega}) \neq \emptyset$(see \cite[Lemma 1.7]{kannan20091torus}). In the sequel, we fix such $v\in W^{P_{\omega}}$. For simplicity we omit the suffix $\omega$.
\begin{proposition}\label{minimalsemisingular} 
Assume that $(G, \omega)$ is different from $(A, \omega_{r})$ with $r=1, n-1$ and $(C, \omega_1)$.
Then the minimal Schubert variety $X_P(v)$ with $X_P(v)^{ss}_T(\mathcal L) \neq \emptyset$ is singular.
\end{proposition}
\begin{proof}

 For the minuscule flag varieties and their Schubert varieties, it is sufficient to restrict to simply laced groups (see \cite[Remark 3.4]{perrin2007small}).
Thus, we need to prove the theorem for $G$ is of type $A$, $D$ and $E$.

In the proof we use the figures, given below,  those are easy to draw for a given expression of $v$ by following the definition of the quiver associated with $v$.

\noindent\underline{\bf Type A:}
Let $ 1 < r < n-1$. By  \cite[Lemma 2.7]{kannan20091torus}), we have that
$v$ is of the form 
 \begin{linenomath*}
$$v=(s_{a_1}\cdots s_1)\cdots(s_{a_r}\cdots s_r)$$\end{linenomath*} for some increasing sequence $a_1<\ldots <a_r$. Then we obtain the quiver associated to $v$ as drawn in Fig. 1.
From Fig. 1, it is clear that the quiver has a non-virtual hole and so we conclude that $X_P(v)$ is singular thanks to Proposition \ref{smoothnesscriterion}.

We repeatedly use Proposition \ref{smoothnesscriterion} and so we do not refer it to each time in the rest of the proof.

\noindent\underline{\bf Type D:} 
If the minuscule fundamental weight is $\omega_1$, then the minimal 
$v$ such that $X(v)^{ss}_T(\mathcal L_1) \neq \emptyset$ is $v=s_ns_{n-1}s_{n-2}\cdots s_2s_1$ (see \cite[Theorem 3.2]{Kannansantosh}). 
Then from Fig. 2, it is clear that there is a hole at the node $n-2$. Hence, we conclude that $X(v)$ is singular. 

\noindent Now we consider the case $\omega_{n-1}$. The proof for $\omega_n$ is similar. 

\noindent We have $\omega_{n-1}=\frac{1}{2}(\alpha_1+2\alpha_2+ \cdots + (n-2)\alpha_{n-2})+\frac{1}{2}n\alpha_{n-1}+\frac{1}{2}(n-2)\alpha_n$ 
and the description of minimal $v \in W^P$ such that $X(v)^{ss}_T(\mathcal L_{n-1}) \neq \emptyset$ is given as follows (see \cite[Theorem 3.2]{Kannansantosh} for more details):

\noindent\underline{\bf Case-1:} $n$ is even. We have
$$v= w_{\frac{n}{2}} \cdots w_2w_1,$$ where $$w_i=\left\{ \begin{array}{ll}
         \tau_is_n & \mbox{if $i \,\, is \,\, even$};\\
         \tau_is_{n-1} & \mbox{if $i \,\, is \,\, odd$},\end{array} \right. $$
with $\tau_i=s_{2i-1}\cdots s_{n-2}$, $i=1,2, \ldots, \frac{n}{2}$.

The quiver associated with this $v$ is shown in Fig. 3.
In this case, the quiver has a non-virtual hole. So, $X(v)$ is singular. 

\noindent\underline{\bf Case-2:} $n$ is odd. We have 	
$$v= w_{[\frac{n}{2}]+1} \cdots w_2w_1,$$ where $$w_i=\left\{ \begin{array}{ll}
         \tau_is_{n-1} & \mbox{if $i \,\, is \,\, even$};\\
         \tau_is_{n} & \mbox{if $i \,\, is \,\, odd$},\end{array} \right.$$

with $$\tau_i=s_{2i-1}\cdots s_{n-2}, i=1,2, \ldots ,\frac{n}{2}+1.$$ The proof in this case is similar to the even case and so we omit the proof here.

\noindent{\bf \underline{Type E:}}

\noindent\underline{Type $(E_6)$ (respectively, $E_7$):} The minuscule fundamental weight is $\omega_1$ (respectively, $\omega_7$) and the minimal 
$v$ such that $X(v)^{ss}_T(\mathcal L_1) \neq \emptyset$ is of the form $$v=s_5s_6s_1s_3s_4s_5s_2s_4s_3s_1 ~\text{(respectively,}~v=s_5s_2s_4s_3s_7s_6s_5s_4s_1s_2s_3s_4s_5s_6s_7);$$
see \cite[p. 3818]{pattanayak2014minimal}.
Then, the quiver associated to $v$ can be described as in Fig. 4 (respectively, Fig. 5).
In both the cases, the quiver contains a non-virtual hole and so $X(v)$ is singular.
 \end{proof}

\noindent The following lemma is a reformulation of Lemma \ref{3.4}.

\begin{lemma}\label{smooth}
% Let $X_P(w)$ be a minuscule Schubert variety and let $v\leq w$ be such that it is minimal and $X_P(v)^{ss}_T(\mathcal L)\neq \emptyset$. Then 
The semistable locus $X_P(w)^{ss}_T(\mathcal L)$ is contained in the smooth locus of $X_P(w)$ if and only if the quiver $Q_v$ is not contained in $Q_{w'}$ for all 
singular components of $X_P(w)$.
\end{lemma}

%\begin{center}
 \begin{minipage}{0.5\textwidth}
\resizebox{7cm}{7.8cm}{
$$\begin{tikzpicture}
	%% Notice in the first vertex is named (v) for the sake of a later edge,
	%% and it also has a label to its left that is the math-mode $v$. 
	\vertex[fill] (1) at (-11, 6)  {};  
	\vertex [fill](2) at (-10,5) {};
	\vertex [fill](22) at (-10,7) [label=above: $a_1$] {};
		\vertex [fill](3) at (-9,4)  {};
		\vertex [fill](32) at (-9,6) {};
		\vertex [fill](42) at (-8,7)  {};
		\vertex [fill](41) at (-8, 5) {};
		\vertex [fill](52) at (-7,8) [label=above: $a_2$ ] {};
		\vertex [fill](51) at (-7,6) {};
		%\vertex[fill] (r-2) at (-8, 3) [label=below: r-2]  {};  
	\vertex [fill](r-1) at (-7,2) {};
		\vertex [fill](r) at (-6,1)  {};
		\vertex [fill](r+1) at (-5.2,2)  {};
		\vertex [fill](r2) at (-6.2,3) {};
	\vertex[fill] (n-3) at (-1,7)  {};
	\vertex[fill] (n-31) at (-1,9.2) [label=above: $a_{r-1}$] {};
	\vertex[fill] (n-312) at (-2,8)  {};
	\vertex[fill] (n-2) at (0,8.2)  [label=below: ] {};
	\vertex[fill] (n-1) at (1,9.4) {};
	\vertex[fill] (n) at (2,10.6) [label=above: $a_r$]{};
	\vertex[fill]  (69) at (-4.5, 8.8) [label=above: $a_3$] {};
	\coordinate  (308) at (-3, 7.4) {};
	\vertex[fill]  (309) at (-6, 7) {};
	\node[style={circle,draw=black, very thick,inner sep=5pt}] at (32) at (-9,6) {};
	\node[style={circle,draw=black, very thick,inner sep=5pt}] at (309) at (-6,7) {};
	\node[style={circle,draw=black, very thick,inner sep=5pt}] at (n-2) at (0, 8.2) {};
	
	\path
	   % Note that the word "path" here isn't used in the graph-theory sense; the \path command
	   % is always used prior to the list of edges; here, coincidentally, they do form an actual path.
		(1) edge (2)
		(3) edge (2)
		(3) edge [dashed] (r-1)
		(n-3) edge (n-2)
		(n-1) edge (n-2)
		(n-1) edge [dashed](n)
		(r+1) edge (n-3)
		(r) edge [dashed] (r+1)
		(r) edge (r-1)
		(r2) edge (n-31)
		(r-1) edge[dashed](r2)
		(r+1) edge (r2)
		(n-2) edge (n-31)
		(n-3) edge (n-312)
		(1) edge[dashed] (22)
		(22) edge (32)
		(32) edge (42)
		(2) edge[dashed] (32)
		(69) edge[dashed] (51)
		(41) edge (51)
		(41) edge[dashed] (3)
		(42) edge (51)
		(69) edge[dashed] (51)
		(69) edge[dashed] (308)
		(32) edge (41)
		(52) edge[dashed] (42)
		(309) edge (52)
		(41) edge [dashed](r2)

	 ;   % This semicolon ends the \path command.
\end{tikzpicture}$$
}

\resizebox{8cm}{0.5cm}{
\begin{tikzpicture}
	%% Notice in the first vertex is named (v) for the sake of a later edge,
	
	%% and it also has a label to its left that is the math-mode $v$. 
	\vertex[fill] (1) at (-11, 0) [label=below: 1]  {};  
	\vertex [fill](2) at (-10,0) [label=below: 2] {};
		\vertex [fill](3) at (-9,0) [label=below: 3 ] {};
		%\vertex[fill] (r-2) at (-8, 0) [label=below: r-2]  {};  
	\vertex [fill](r-1) at (-7.8,0) [label=below: r-1] {};
		\vertex [fill](r) at (-6.8,0) [label=below: r ] {};
	\vertex[fill] (n-3) at (-2.4,0) [label=below: $a_{r-1}$] {};
	\vertex[fill] (n-2) at (-1.2,0)  [label=below: ] {};
	\vertex[fill] (n-1) at (0,0) [label=below: ]{};
	\vertex[fill] (n) at (1,0) [label=below: $a_r$]{};
	\coordinate (n+1) at (2,0) [label=below:]{};
	
	\path
	   % Note that the word "path" here isn't used in the graph-theory sense; the \path command
	   % is always used prior to the list of edges; here, coincidentally, they do form an actual path.
		(1) edge (2)
		(3) edge (2)
		(3) edge [dashed] (r-1)
		(r-1) edge (r)
		 (r) edge [dashed] (n-3)
		(n-3) edge (n-2)
		(n-1) edge (n-2)
		(n-1) edge[dashed] (n)
		(n) edge[dashed] (n+1)
	
	 ;   % This semicolon ends the \path command.
\end{tikzpicture}
}
$$ ~Fig.~1.~ \mbox{Type A}.$$
\end{minipage}
%\end{center}
\hfil
%\begin{center}
 \begin{minipage}{0.45\textwidth}
\resizebox{7cm}{6cm}{
\begin{tikzpicture}
	%% Notice in the first vertex is named (v) for the sake of a later edge,
	%% and it also has a label to its left that is the math-mode $v$. 
	\vertex[fill] (1) at (-5, 6) [label=below: 1]  {};  
	\vertex [fill](2) at (-4,5) [label=below: 2] {};
	\vertex [fill](22) at (-3,4) [label=below: ] {};
	\vertex[fill] (3) at (-1,2) [label=below: n-4] {};
	\vertex[fill] (4) at (0,1)  [label=below: n-3] {};
	\vertex[fill] (5) at (1,0) [label=below: n-2]{};
	\vertex[fill] (6) at (2,1)  [label= right: n-1]{};
	\vertex[fill] (7) at (3,2.5)  [label=right: n]{};
	
	\node[style={circle,draw=black, very thick,inner sep=5pt}] at (5) at (1,0) {};
	
	\path
	   % Note that the word "path" here isn't used in the graph-theory sense; the \path command
	   % is always used prior to the list of edges; here, coincidentally, they do form an actual path.
		(1) edge (2)
		(22) edge (2)
		(22) edge [dashed] (3)
		(5) edge (4)
		(4) edge (3)
		(5) edge (6)
		(5) edge (7)

	 ;   % This semicolon ends the \path command.
\end{tikzpicture}
}
\resizebox{7cm}{2cm}{
\begin{tikzpicture}
	%% Notice in the first vertex is named (v) for the sake of a later edge,
	%% and it also has a label to its left that is the math-mode $v$. 
	\vertex[fill] (1) at (-5, 0) [label=below: 1]  {};  
	\vertex [fill](2) at (-4,0) [label=below: 2] {};
		\vertex [fill](22) at (-3,0) [label=below: ] {};
	\vertex[fill] (3) at (-1,0) [label=below: n-4] {};
	\vertex[fill] (4) at (0,0)  [label=below: n-3] {};
	\vertex[fill] (5) at (1,0) [label=below: n-2]{};
	\vertex[fill] (6) at (2,1)  [label= n-1]{};
	\vertex[fill] (7) at (3,-1)  [label= n]{};
	
	\path
	   % Note that the word "path" here isn't used in the graph-theory sense; the \path command
	   % is always used prior to the list of edges; here, coincidentally, they do form an actual path.
		(1) edge (2)
		(22) edge (2)
		(22) edge[dashed] (3)
		(3) edge (4)
		(5) edge (4)
		(5) edge (6)
		(5) edge (7)
	 ;   % This semicolon ends the \path command.
\end{tikzpicture}
}

$$ ~Fig.~2.~ \mbox{Type (D, $\omega_1$)}.$$

\end{minipage}%\end{center}

\begin{center}
 \begin{minipage}{0.45\textwidth}
\resizebox{7cm}{6cm}{
\begin{tikzpicture}
	%% Notice in the first vertex is named (v) for the sake of a later edge,
	%% and it also has a label to its left that is the math-mode $v$. 
	\vertex[fill] (1) at (-5, 6) [label=below: 1]  {};  
	\vertex [fill](2) at (-4,5) [label=below: 2] {};
	\vertex [fill](22) at (-3,4) [label=below: ] {};
	\vertex [fill](222) at (-2,5) [label=below: ] {};
	\vertex[fill] (3) at (-1,2) [label=below: n-4] {};
	\vertex [fill](32) at (-3, 6) [label=below: ] {};
	\vertex[fill] (4) at (0,1)  [label=below: n-3] {};
	\vertex[fill] (5) at (1,0) [label=below: n-2]{};
	\vertex[fill] (6) at (2,1)  [label= right: n-1]{};
	\vertex[fill] (7) at (3,2.5)  [label=right: n]{};
	\vertex[fill] (52) at (1,2)  {};
	\vertex[fill] (42) at (0,3)  {};
	\node[style={circle,draw=black, very thick,inner sep=5pt}] at (2) at (-4,5) {};
	\coordinate (9) at (1,4) {};
	\coordinate (8) at (0,5) {};
	\coordinate (10) at (2,4) {};
	\path
	   % Note that the word "path" here isn't used in the graph-theory sense; the \path command
	   % is always used prior to the list of edges; here, coincidentally, they do form an actual path.
		(1) edge (2)
		(22) edge (2)
		(32) edge (2)
		(32) edge  (222)
		(222) edge (22)
		(22) edge [dashed] (3)
		(5) edge (4)
		(4) edge (3)
		(5) edge (6)
		(52) edge (7)
		(42) edge[dashed]  (222)
		(42) edge (3)
		(42) edge (6)
		(52) edge (4)
		(5) edge (7)
		(222) edge [dashed] (8)
		(9) edge [dashed] (42)
		(10) edge [dashed] (52)

	 ;   % This semicolon ends the \path command.
\end{tikzpicture}
}
\resizebox{7cm}{2cm}{
\begin{tikzpicture}
	%% Notice in the first vertex is named (v) for the sake of a later edge,
	%% and it also has a label to its left that is the math-mode $v$. 
	\vertex[fill] (1) at (-5, 0) [label=below: 1]  {};  
	\vertex [fill](2) at (-4,0) [label=below: 2] {};
		\vertex [fill](22) at (-3,0) [label=below: ] {};
	\vertex[fill] (3) at (-1,0) [label=below: n-4] {};
	\vertex[fill] (4) at (0,0)  [label=below: n-3] {};
	\vertex[fill] (5) at (1,0) [label=below: n-2]{};
	\vertex[fill] (6) at (2,1)  [label= n-1]{};
	\vertex[fill] (7) at (3,-1)  [label= n]{};
	
	\path
	   % Note that the word "path" here isn't used in the graph-theory sense; the \path command
	   % is always used prior to the list of edges; here, coincidentally, they do form an actual path.
		(1) edge (2)
		(22) edge (2)
		(22) edge[dashed] (3)
		(3) edge (4)
		(5) edge (4)
		(5) edge (6)
		(5) edge (7)
	 ;   % This semicolon ends the \path command.
\end{tikzpicture}
}
$$
Fig.~ 3. ~\mbox{Type $(D, \omega_{n-1})$}$$
\end{minipage}
\end{center}

\begin{minipage}{0.45\textwidth}
\[\begin{tikzpicture}
	%% Notice in the first vertex is named (v) for the sake of a later edge,
	%% and it also has a label to its left that is the math-mode $v$. 
	\vertex[fill] (1) at (-3, 0)   {};  
	\vertex [fill](3) at (-2,1)  {};
	\vertex[fill] (4) at (-1,2) {};
	\vertex[fill] (5) at (0,3)  {};
	\vertex[fill] (6) at (1,4) {};
	\vertex[fill] (2) at (-1,3)  {};
	\vertex[fill] (12) at (-3, 6)  {};  
	\vertex [fill](32) at (-2,5)  {};
	\vertex[fill] (42) at (-1,4) {};
	\vertex[fill] (52) at (0,5)  {};	
	\node[style={circle,draw=black, very thick,inner sep=5pt}] at (42) at (-1,4) {};
	\path
	   % Note that the word "path" here isn't used in the graph-theory sense; the \path command
	   % is always used prior to the list of edges; here, coincidentally, they do form an actual path.
		(1) edge (3)
		(4) edge (2)
		(3) edge (4)
		(4) edge (5)
		(5) edge (6)
		(12) edge (32)
		(42) edge (32)
		(42) edge (2)
		(42) edge (5)
		(52) edge (6)
		(52) edge (42)

	 ;   % This semicolon ends the \path command.
\end{tikzpicture}\]
\vspace{1cm}

\[\begin{tikzpicture}
	%% Notice in the first vertex is named (v) for the sake of a later edge,
	%% and it also has a label to its left that is the math-mode $v$. 
	\vertex[fill] (1) at (-3, 0) [label=below: 1]  {};  
	\vertex [fill](3) at (-2,0) [label=below: 3] {};
	\vertex[fill] (4) at (-1,0) [label=below: 4] {};
	\vertex[fill] (5) at (0,0)  [label=below: 5] {};
	\vertex[fill] (6) at (1,0) [label=below: 6]{};
	\vertex[fill] (2) at (-1,1)  [label= 2]{};
	\path
	   % Note that the word "path" here isn't used in the graph-theory sense; the \path command
	   % is always used prior to the list of edges; here, coincidentally, they do form an actual path.
		(1) edge (3)
		(3) edge (4)
		(5) edge (4)
		(2) edge (4)
		(5) edge (6)
	 ;   % This semicolon ends the \path command.
\end{tikzpicture}\]
$$Fig.~4. ~\mbox{Type $E_6$}$$
\end{minipage}
\begin{minipage}{0.45\textwidth}
 \[\begin{tikzpicture}
	%% Notice in the first vertex is named (v) for the sake of a later edge,
	%% and it also has a label to its left that is the math-mode $v$. 
	\vertex[fill] (1) at (-3, 7)   {};  
	\vertex [fill](3) at (-2,6)  {};
	\vertex[fill] (4) at (-1,5) {};
	\vertex[fill] (5) at (0,4)  {};
	\vertex[fill] (6) at (1,3) {};
	\vertex[fill] (7) at (2,2) {};
	\vertex[fill] (2) at (-1,6)  {};
	\vertex [fill](32) at (-2,8)  {};
	\vertex [fill](52) at (0,8)  {};
	\vertex [fill](53) at (0,10)  {};
	\vertex [fill](42) at (-1,7)  {};
	\vertex [fill](43) at (-1,9)  {};
	\vertex [fill](22) at (-1,10)  {};
	\vertex [fill](62) at (1,9)  {};
	\node[style={circle,draw=black, very thick,inner sep=5pt}] at (43) at (-1,9) {};
	\path
	   % Note that the word "path" here isn't used in the graph-theory sense; the \path command
	   % is always used prior to the list of edges; here, coincidentally, they do form an actual path.
		(1) edge (3)
		(3) edge (4)
		(5) edge (4)
		(2) edge (4)
		(5) edge (6)
		(6) edge (7)
		(42) edge (52)
		(32) edge (42)
		(3) edge (42)
		(32) edge (1)
		(2) edge (42)
		(52) edge (43)
		(32) edge (43)
		(62) edge (52)
		(62) edge (53)
		(43) edge (53)
		(22) edge (43)
		
	 ;   % This semicolon ends the \path command.
\end{tikzpicture}\]
\[\begin{tikzpicture}
	%% Notice in the first vertex is named (v) for the sake of a later edge,
	%% and it also has a label to its left that is the math-mode $v$. 
	\vertex[fill] (1) at (-3, 0) [label=below: 1]  {};  
	\vertex [fill](3) at (-2,0) [label=below: 3] {};
	\vertex[fill] (4) at (-1,0) [label=below: 4] {};
	\vertex[fill] (5) at (0,0)  [label=below: 5] {};
	\vertex[fill] (6) at (1,0) [label=below: 6]{};
	\vertex[fill] (7) at (2,0) [label=below: 7]{};
	\vertex[fill] (2) at (-1,1)  [label= 2]{};
	\path
	   % Note that the word "path" here isn't used in the graph-theory sense; the \path command
	   % is always used prior to the list of edges; here, coincidentally, they do form an actual path.
		(1) edge (3)
		(3) edge (4)
		(5) edge (4)
		(2) edge (4)
		(5) edge (6)
		(6) edge (7)
	 ;   % This semicolon ends the \path command.
\end{tikzpicture}\]
$$Fig.~5. ~\mbox{Type $E_7$}$$
\end{minipage}

\noindent The following result generalizes Proposition \ref{semi-singular} to minuscule Schubert varieties.

\begin{theorem} \label{quiver1}% We keep the notation
 The semistable locus $X_P(w)^{ss}_T(\mathcal L)$ is contained in the smooth locus of $X_P(w)$ if and only 
 if the quiver $Q_v$ contains all the essential holes of $Q_w$. 
\end{theorem}
\begin{proof}
 Since the Schubert varieties in projective spaces are always smooth, we only need to consider minuscule Grassmannians $G/P$  that are not projecive spaces; see Remark \ref{exclude}. 
 Therefore, we need to prove the theorem for $(G, \omega)$ different from $(A, \omega_{r})$ with $r=1, n-1$ and $(C, \omega_1)$. 
 
 We first assume that  $X_P(w)^{ss}_T(\mathcal L)$ is contained in the smooth locus of $X_P(w)$. 
 By Lemma \ref{smooth}, the quiver $Q_v$ is not contained in $Q_{w'}$ for any singular component $X_P(w')$ of $X_P(w)$. 
 Let $i$ be an essential hole of $Q_{w}$ and is not contained in $Q_{v}$. We claim that the irreducible component $X_P(w^i)$ contains $X_P(v)$.
 To that end, it is enough to show that the quiver $Q_{w, i}$ contains $Q_{v}$. By Proposition \ref{minimalsemisingular},  the Schubert variety $X_P(v)$ is singular. So, the quiver $Q_{v}$ contains an essential hole. Recall that by the definition of essential hole the subquiver $Q_{w}^i$ contains no hole. Hence, the quiver $Q_{v}$ must contained in the quiver $Q_{w, i}$. This implies that the claim follows.
 
 Conversely, suppose the quiver $Q_v$ contains all essential holes of $Q_w$. Then, by similar arguments as above we can easily see that the quiver $Q_{v}$ is not contained 
 in any of the quivers $Q_{w, i}$ for all essential holes of $Q_{w}$. Hence, we conclude the proof.
  \end{proof}

\section{A special case for $G=SL(n, \mathbb C)$}
\subsection{}\label{main} Let $P=P_1\cap P_{n-1}$ be the parabolic subgroup of $G=SL(n, \mathbb C)$ corresponding to $\alpha_0=\omega_1+\omega_{n-1}$. 
Note that the homogeneous line bundle $\mathcal L_{\alpha_0}$ is very ample on the flag variety $G/P$.

The main results of the section are the following:
Let $T\subseteq B\subseteq P$ be a maximal torus. Fix $w\in W^P$ such that $X_P(w)^{ss}_{T}(\mathcal L_{\alpha_0})$ is nonempty. 
Then we prove:
\begin{theorem}\label{projectivenormality} Let $\mathcal M$ be the descent of $\mathcal L_{\alpha_0}$ to the quotient $X_P(w)^{ss}_{T}(\mathcal L_{\alpha_0})\sslash T$. Then,
 the polarized variety $(X_P(w)^{ss}_{T}(\mathcal L_{\alpha_0})\sslash T, \mathcal M)$ is projectively normal.
\end{theorem}

\noindent We deduce the following:
\begin{corollary}\label{quotient}
 The quotient $X_P(w)^{ss}_T(\mathcal L_{\alpha_0})\sslash T$  is isomorphic to a projective space.
\end{corollary}

We prove the results by using the theory of standard monomials on Schubert varieties. We first observe that it is sufficient to prove these results for the full flag variety $G/B$. Indeed, since $\pi: G/B\to G/P$ is a fibration with fiber $P/B$, we have the following equality:
\begin{equation}\label{tofullflag}
    H^0(G/B, \pi^*\mathcal L_{\alpha_0})=H^0(G/P, \mathcal L_{\alpha_0})
\end{equation}

Let $w\in W^P$. Then the restriction of $\pi$ to $X(ww_{0,P})\to X_P(w)$ is also a $P/B$-fibration, where $w_{0,P}$ is the longest element in $W_P$. Hence, it follows that $H^0(X(ww_{0,P}), \pi^*\mathcal L_{\alpha_0})=H^0(X_P(w), \mathcal L_{\alpha_0})$
(see \cite[p.104]{LMS} for more details).
So from now on, we work with the Schubert varieties inside the full flag variety $G/B$.

We thus, start by recalling some basic notions on standard monomial theory: see \cite[p. 120]{lakshmibaigrassmannian} for more details.
\subsection{Standard monomial theory on Schubert varieties}\label{standard1}
Let us fix a dominant integral weight $\lambda$. For this $\lambda$ we associate a Young diagram with at most $n-1$ columns. Let $m_i$ be the number of rows 
of $\lambda$ of length $i$, for $1\leq i\leq n-1$. We denote the rows of $\lambda$ from bottom to top of length $i$ by 
\begin{linenomath*}$$ \tau_{i,1}, \tau_{i,2}, \ldots ,\tau_{i,m_i}.$$\end{linenomath*}
Thus, rows of $\lambda$ from bottom to top are denoted by \begin{linenomath*}
$$ \tau_{1,1}, \tau_{1,2}, \ldots ,\tau_{1,m_1}, \ldots  ,\tau_{(n-1),1}, \tau_{(n-1),2}, \ldots ,\tau_{(n-1),m_{n-1}}.$$\end{linenomath*}
Note that if $m_i=0$ for some $i$, then $\{\tau_{i,j}: 1\leq j \leq m_i\}$ is empty.
Let $\Lambda$ be a standard tableau of shape $\lambda$ with entries from $\{1, \ldots, n\}$. Then, a row $\tau_{i,j}$ can be thought of as an element 
of $I_{i, n}:=\{1\leq j_1<\cdots j_i\leq n\}$. For a given $$w=(w(1), \ldots, w(n))=(j_1,\ldots, j_n)\in S_n,$$ let $\pi_i(w)=(j_1, \ldots, j_i)\in I_{i.j}$, where the elements
$j_1,j_2,\ldots, j_i$ are rearranged in increasing order.

{\it \bf Young tableau on $X(w)$}: Consider a tableau $\Lambda$ of shape $\lambda$ with entries in $\{1, \ldots, n\}$. We say that $\lambda$ is 
a {\it Young tableau on $X(w)$} if \begin{linenomath*}$$\pi_i(w)\geq \tau_{i,j}~\mbox{for all}~1\leq j \leq m_i \,\, \mbox{and for all} \,\, 1\leq i\leq n-1.$$\end{linenomath*}
Note that in this case $p_{\Lambda}|_{X{w}}\neq 0$. In addition, we say that $\Lambda$ is a {\it standard tableau} on $X(w)$ if there exists a sequence of elements in $S_n$ 
\begin{linenomath*}$$\phi_{1,1}, \ldots , \phi_{1, m_1}, \ldots, \phi_{(n-1),1}, \ldots , \phi_{(n-1), m_{(n-1)}}$$\end{linenomath*} such that 
\begin{linenomath*}$$X(w)\supseteq X(\phi_{1,1})\supseteq\cdots \supseteq X(\phi_{1, m_1})\supseteq X(\phi_{2,1})\supseteq \cdots \supseteq X(\phi_{(n-1), m_{n-1}})$$\end{linenomath*}
and \begin{linenomath*} $$\pi_i(\phi_{i,j})=\tau_{i,j}~\mbox{for all}~ 1\leq j\leq m_i ~\mbox{and}~ 1\leq i \leq n-1.$$\end{linenomath*}

Let $\mathcal L_{n, \lambda}^{w}$ be the set of Young tableau of shape $\lambda$ on $X(w)$. 
We have the following theorem (see \cite[Theorem 8.2.9, p.121]{lakshmibaigrassmannian}).
\begin{theorem}
 The set $\{p_{\Lambda} ~;~ \Lambda\in \mathcal L_{n, \lambda}^{w}\}$  forms a basis for $H^0(X(w), L_{\lambda})$.
\end{theorem}

We now recall the definition of weight of a standard Young tableau $\Lambda$ (see \cite[Sec. 2]{littelmann}). For a positive integer $1\leq i\leq n$, we denote $c_{\Lambda}(i)$ by
the number of boxes of $\Lambda$ containing the integer $i$. Let $\epsilon_i: T\to \mathbb G_m$  
be the character defined as $\epsilon_i(diag(t_1,\ldots, t_n))=t_i$. Recall that the $i^{th}$ fundamental weight is given by 
$$\omega_i=\epsilon_1+ \cdots +\epsilon_i.$$ Then we define the weight of $\Lambda$ as follows 
\begin{linenomath*}$$wt(\Lambda):=c_{\Lambda}(1)\epsilon_1+ \cdots + c_{\Lambda}(n)\epsilon_n.$$\end{linenomath*}

\noindent The following result is well known and for completeness we give a proof here.
\begin{lemma}\label{6.2}
The section $p_{\Lambda}\in H^0(X(w), \mathcal L_{\lambda})$ is $T$-invariant if and only if  all $i$ appear in $\Lambda$ and with equal number of times. 
\end{lemma}
\begin{proof}
 Recall that the action of $T$ on $H^0(X(w), \mathcal L_{\lambda})$ is given by \begin{linenomath*} $$(t_1, \ldots, t_n)\cdot p_{i_1,i_2,\ldots, i_r}=(t_{i_1}\cdots t_{i_r})^{-1}p_{i_1,i_2, \ldots, i_r}.$$ \end{linenomath*}
Since $p_{i_1,\ldots, i_r}$ is the dual of $e_{i_1}\wedge \cdots \wedge e_{i_r}$, the weight of $p_{i_1,\ldots, i_r}$ is $-(\epsilon_{i_1}+\cdots +\epsilon_{i_r})$.
Thus, the weight of $p_{\Lambda}$ is same as negative of weight of the tableau $\Lambda$.
Therefore, it follows  that the section $p_{\Lambda}$ is $T$-invariant if and only if the weight of $\Lambda$ is zero.
Since the weight of $\Lambda$ is $$\sum_{i=1}^{n}c(i)\epsilon_i$$ and $\sum_{i=1}^{n}\epsilon_i=0$, we conclude that the section $p_{\Lambda}$ is $T$-invariant if and only if 
$c(i)=c(j)$ for all $1\leq i, j \leq n$ and this proves the lemma.
\end{proof}

\noindent We finish this section with the following observation.

\begin{lemma}
 Let $v\leq w$ and let $P_{\Lambda}$ be a standard monomial on $X(v)$. Then $P_{\Lambda}$ is also a standard monomial on $X(w)$.
\end{lemma}
\begin{proof} It is well known that we have the following surjective homomorphism given by restrictions of sections\begin{linenomath*}
$$H^0(X(w), \mathcal L_{\alpha_0})\to H^0(X(v), \mathcal L_{\alpha_0}).$$\end{linenomath*}
Since the standard monomials on $X(v)$ form a basis for the space $H^0(X(v), \mathcal L_{\alpha_0})$, it follows that the section
$P_{\Lambda}$ can be lifted to $X(w)$ and it is standard on $X(w)$ as we can choose the same admissible sequence 
that was taken for $X(v)$.
\end{proof}

\subsection{Homogeneous coordinate ring.}\label{coordinatering}
In this subsection, we study the homogeneous coordinate ring of the 
torus quotient of Schubert varieties explicitly using standard monomials.

Here, we always assume that $X_P(w)^{ss}_{T}(\mathcal L_{\alpha_0})\neq \emptyset$.
Recall that the homogeneous coordinate ring of the quotient $X_P(w)^{ss}_T(\mathcal L_{\alpha_0})\sslash T$ is given by 
$$\bigoplus_{m\geq 0}H^0(X_P(w), \mathcal L_{m\alpha_0})^T$$
(see \cite[Theorem 3.14, p.76]{GITmumford}).

As we discussed in Sec.\ref{main}, we reduce %this
to the case of the full flag variety $G/B$.
Therefore, from now on we consider Schubert varieties inside $G/B$. By abuse of notation,  we also use $w$ to parametrise a Schubert variety in $G/B$.  That is, we consider $w$ as an element of the symmetric group $S_{n}$.

{\it The idea is to start with a minimal $v$ associated to a Schubert variety admitting semistable points and extend it to $w_0$, the longest element of the Weyl group. In this way,  we obtain the results for all Schubert varieties admitting semistable points. In this process, we use some nice structural properties between the Young tableaux and the standard monomials on Schubert varieties. In our proofs we use the results from Sec.\ref{standard1} freely without mentioning them explicitly.}

We start with the description of minimal Schubert varieties in $G/B$ from \cite[Theorem 4.2]{Kannansantosh}.
\begin{theorem}\label{santosh}
An element $v \in W=S_{n}$ is minimal with $X(v)^{ss}_{T}(\mathcal L_{\alpha_0})\neq \emptyset$ if and only if $v$ is of the form 
 \begin{enumerate}
  \item $v=s_{i+1}\cdots s_{n-1}s_i\cdots s_1$ for some $1\leq i\leq n-1$; or
\item $v=s_is_{i-1}\cdots s_{1}s_{i+1}\cdots s_{n-1}$ for some $1\leq i\leq n-1$.
 \end{enumerate}
  \end{theorem}

Here, notice that for $i=n-1, n-2$ the element $v$ is $s_{n-1}\cdots \cdots s_1$.

  \begin{remark}
   Theorem 4.2 of \cite{Kannansantosh} is stated for Coxeter elements but by \cite[Lemma 2.1]{Kannansantosh} we need to have $v(\alpha_0)<0$. This happens only if all the simple reflections appear in $v$. 
  \end{remark}
  
 Let $v \in W=S_{n}$ be minimal such that $X(v)^{ss}_{T}(\mathcal L_{\alpha_0})\neq \emptyset$. Now for $w \geq v$ we describe explicitly the space $H^0(X(w), \mathcal L_{\alpha_0})^T$ of $T$-invariant sections on the Schubert variety $X(w)$. 
  
Recall that $\alpha_0=\omega_1+\omega_{n-1}$. Let $\Lambda$ be a tableau of shape $m\alpha_0$, $m \in \mathbb N$.  From now on, we denote $\Lambda$ by the sequence (as above)
\begin{linenomath*} $$ \Lambda= \tau_{1,1},\ldots, \tau_{1,m}, \tau_{(n-1),1},\ldots ,\tau_{(n-1), m}.$$\end{linenomath*}

 The corresponding standard monomial is denoted by $$p_{\tau_{1,1}}\cdots p_{\tau_{1,m}}p_{\tau_{(n-1), 1}} \ldots  p_{\tau_{(n-1), m}}.$$
 Let $w=(w(1), w(2), \ldots, w(n))$ be the one-line notation for $w\in S_n$. 
\begin{lemma}\label{6.4} Let $v$ be as in Theorem \ref{santosh}. Then,

\begin{enumerate}
 \item  The section $p_{\Lambda}\in H^0(X(v), \mathcal L_{\alpha_0})$ is $T$-invariant if and only if $$\tau_{1,1}=\{v(1)\}~\text{and}~\tau_{(n-1), 1}=\{1,2, \ldots, v(1)-1, v(1)+1, \ldots, n\}.$$
\item The section $p_{\Lambda}\in  H^0(X(v), \mathcal L_{m\alpha_0})$ is $T$-invariant  if and only if $$p_{\Lambda}=(p_{\tau_{1,1}}p_{\tau_{(n-1), 1}})^m$$ with $\tau_{1,1}$ and $\tau_{(n-1), 1}$ are as in (1).
\end{enumerate}

\end{lemma}
\begin{proof} Using the definition of standard monomials on Schubert varieties and Lemma \ref{6.2} the proof follows. 
\end{proof}

Before going to the proof of the main results in this section, we see an example  
that also serves as a guiding path to prove the main results. 

\subsubsection{Example}
Let  $G=SL(7, \mathbb C)$ and $i=3$ in Theorem \ref{santosh}. Then 
$v=s_4s_5s_6s_3s_2s_1$ and in one line notation, $v$ is represented by $(5,1,2,3,6,7,4)$. 
We have a Young tableau $\Lambda$ given by \begin{linenomath*}
\[\Lambda = \begin{ytableau}
 1 & 2 & 3 & 4 & 6& 7 \\
  5
\end{ytableau}
\]\end{linenomath*}
Thus, we have $\tau_{n-1, 1}= (1,2,3,4,6,7)$ and $\tau_{1,1}=(5)$. We take $\phi_{n-1,1}=(1,2,3,4,6,7,5)$ and $\phi_{1,1}=(5,1,2,3,6,7,4)$. Then it follows
$\pi_{i}(\phi_{i,1})=\tau_{i,1}$ for $i=1, n-1$. 
Therefore, we have a sequence \begin{linenomath*} $$\phi_{n-1,1}\leq \phi_{1,1}\leq w$$\end{linenomath*} such that $\Lambda$ is a standard tableau on $X(v)$.
Also, note that  if $l<5$ then, any tableau with $\tau_{1 1}=l$ is not a standard tableau on $X(v)$ as there is no sequence satisfying the conditions stated above. 
Hence, we conclude that the equality $dim(H^0(X(v), \mathcal L_{\alpha_0})^T)=1$ holds. 

Now we extend $v$ to $w_0$ by multiplying simple reflections step by step and observe how the dimension changes:
Note that $s_2v=(5,2,1,3,6,7,4)$. Clearly, \begin{linenomath*}
$$\Lambda= \begin{ytableau}
 1 & 2 & 3 & 4 & 6& 7 \\
  5
\end{ytableau}$$ \end{linenomath*}is a standard Young tableau on $X(s_2v)$. Now we show that there are no other standard Young tableau on $X(s_2v)$.
Let  \begin{linenomath*}$$\Lambda'= \begin{ytableau}
 1 & 2 & 3 & 5 & 6& 7 \\
  4
\end{ytableau}.$$ \end{linenomath*}Then we must have $$\phi_{n-1, 1}=(1,2,3,5,6,7,4)\leq \phi_{1,1}=(4,...,)\leq s_2v=(5,2,1,3,6,7,4).$$ 
Since  $\phi_{1,1}\leq s_2v$ we have $\phi_{1,1}=(4,2,1,3,\ldots)$. This implies  \begin{linenomath*}
$$\phi_{1,1}\ngeq \phi_{n-1,1},~\mbox{ as} ~(4,2,1,3)\uparrow \ngeq (1,2,3,5)\uparrow$$\end{linenomath*}
(here $\uparrow$ denote the increasing order). 
Hence, there is no such $\phi_{1,1}$ and so $\Lambda'$ is not standard on $X(s_2v)$. Similarly, we can see that 
any tableau \begin{linenomath*}$$\Lambda'=\begin{ytableau}
 1 &  &  &  & &  \\
  l
\end{ytableau},$$\end{linenomath*} with $l<5$ is not standard on $X(s_2v)$.
Hence, it is clear $dim(H^0(X(s_2v), \mathcal L_{\alpha_0})^T)=1$.

By the similar arguments as above, for $w=s_3s_2w, s_4s_3s_2w$ and $s_5s_4s_3s_2w$, we see that $dim(H^0(X(w), \mathcal L_{\alpha_0})^T)=1$.

Let $w=s_6s_5s_4s_3s_2v=(5,2,3,6,7,4,1)$. In this case we show that any Young tableau 
\begin{linenomath*} $$\Lambda= \begin{ytableau}
 1 &  &  &  & &  \\
  l
\end{ytableau}$$ \end{linenomath*} with $2\leq l \leq 5$ is standard on $X(w)$. We now give the admissible sequence $\phi_{n-1,1}\leq \phi_{1,1}\leq w$ for each $l$.
\begin{itemize}
 \item 
For $l=4$; $\phi_{n-1,1}=(1,2,3,5,6,7,4)$ and $\phi_{1,1}=(4,2,3,6,7,5,1)$.
\item For $l=3$; $\phi_{n-1,1}=(1,2,4,5,6,7,3)$ and $\phi_{1,1}=(3,2,5,6,7,4,1)$.

\item For $l=2$; $\phi_{n-1,1}=(1,3,4,5,6,7,2)$ and $\phi_{1,1}=(2,5,3,6,7,4,1)$.
\end{itemize}
Hence, we conclude that $dim(H^0(X(w), \mathcal L_{\alpha_0})^T)=4$.

By the similar arguments as above, for $w$ of the form: 
\begin{itemize}
 \item 
$s_2s_6s_5s_4s_3s_2v$, 
\item $s_3s_2s_6s_5s_4s_3s_2v$,
\item $s_4s_3s_2s_6s_5s_4s_3s_2v$, 
\item $s_5s_4s_3s_2s_6s_5s_4s_3s_2v$,
\item $s_2s_5s_4s_3s_2s_6s_5s_4s_3s_2v$,
\item $s_2s_3s_5s_4s_3s_2s_6s_5s_4s_3s_2v$ and 
\item $s_4s_3s_2s_5s_4s_3s_2s_6s_5s_4s_3s_2v=(5,6,7,4,3,2,1)$,

\end{itemize}
we can easily show that $dim(H^0(X(w), \mathcal L_{\alpha_0})^T)=4$. 

Let $w'=(5,6,7,4,3,2,1)$ and $w=s_2w'$ or $w=s_3s_2w'$. Then we have for $l=6$, the admissible sequence on $X(w)$ is given by 
$\phi_{n-1,1}=(1,2,3,4,5,7,6)$ and $\phi_{1,1}=(6,5,7,4,3,2,1)$.
It follows that $dim(H^0(X(w), \mathcal L_{\alpha_0})^T)=5$. 

Let $w=s_2s_3s_2w'=(7,6,5,4,3,2,1)=w_0$. For $l=7$, the admissible sequence is given by \begin{linenomath*}
$$\phi_{n-1,1}=(1,2,3,4,5,6,7) ~\mbox{and}~ \phi_{1,1}=(7,6,5,4,3,2,1).$$ \end{linenomath*}
This implies, we have  $dim(H^0(X(w), \mathcal L_{\alpha_0})^T)=6$.

\subsection{} We follow the similar strategy as in the above example to prove the results in this section.

%To prove the main results in this section we need a series of lemmas:

In the following, we denote by "{\bf Minimal $v$}" for a minimal $v$ such that $X(v)^{ss}_{T}(\mathcal L_{\alpha_0})\neq \emptyset$ as in Theorem \ref{santosh}.  For $w \geq v$, we extend $v$ to $w$ by multiplying simple reflections. We say such $w$ is an {\bf extension} of $v$ and we denote it by "{\bf Extension $w$}".  

\subsubsection{}\underline{{\bf Case-A:}} We take $v$ from Theorem \ref{santosh}(1). Let $1\leq i\leq n-1$. 
%\noindent\underline{Notation:} 
\begin{enumerate}
 \item  ({\bf Minimal $v$}) Define  $v=v_{0, i}:=s_{i+1}\cdots s_{n-1}s_i\cdots s_1$.
 \item ({\bf Extension $w$}) Define $v_{k, j}:= \begin{cases} s_j\cdots s_2 v_{k-1, n-(k-1)}, ~&\text{if}~ 1\leq k\leq i ~ \text{and}~ 2\leq j\leq n-k,\\
                   v_{i, n-i} ~~\text{if}~ k=i+1\\
                   s_j\cdots s_1 v_{k-1, n-(k-1)} ~&\text{if}~ k>i+1 ~\text{and}~ 1\leq j\leq n-k.	
                  \end{cases}
 	 $
\end{enumerate}
\underline{\bf Case-A.1:} $i\neq n-1, n-2$.

\begin{lemma} \label{generators1}
Let $k=1$ and $2\leq j\leq n-2$. Then, 
\begin{enumerate}
 \item We have $dim(H^0(X(v_{1, j}), \mathcal L_{\alpha_0})^T=1$.
 \item The section $P_{\Lambda}\in H^0(X(v_{1, j}), \mathcal L_{\alpha_0})$ is $T$-invariant if and only if 
 \begin{equation}\label{eqinva}
\tau_{1,1}=i+2 ~\text{and}~\tau_{n-1, 1}=1, \ldots ,i+1, i+3, \ldots, n.     
 \end{equation}
\end{enumerate}
\end{lemma}

\begin{proof}
 Since $X(v_{0, i})\subset X(v_{1, j})$ and $T$ is reductive, the restriction map  \begin{linenomath*}$$H^0(X(v_{1, j}), \mathcal L_{\alpha_0})^T\longrightarrow H^0(X(v_{0, i}), \mathcal L_{\alpha_0})^T$$ \end{linenomath*}
 is surjective. It follows that $dim(H^0(X(v_{1, j}), \mathcal L_{\alpha_0})^T\geq 1$.
 Let $\Lambda$ be a tableau with $\tau_{1,1}$ and $\tau_{n-1, 1}$ as in (\ref{eqinva}). Then, the section $p_{\Lambda}\in H^0(X(v_{1, j}), \mathcal L_{\alpha_0})$ is $T$-invariant.
 
 To prove the lemma it is enough to show that $\Lambda$ as above is the only possibility. 
 We have $$v_{1,j}=(i+2, 2, 3, \ldots, j, 1, j+1, \ldots, n, i+1).$$
  Let $\tau_{1,1}=l$ and $l<i+2$ (as $\Lambda$ to be a Young tableau on $X(v_{1,j})$). 
  Note that $$\tau_{n-1,1}=(1,2,\ldots,l-1,l+1,\ldots, n).$$ We need to construct a sequence 
 \begin{equation}\label{phisequence}
  v_{1,j}\geq \phi_{1,1}\geq \phi_{n-1,1}
 \end{equation}
 such that $\pi_{n-1}(\phi_{n-1,1})=\tau_{n-1,1}$ and $\pi_{1}(\phi_{1,1})=\tau_{1,1}$.
 Then, we obtain \begin{linenomath*}
 $$\phi_{n-1,1}=(1,2,\ldots, l-1,l+1, \ldots, n, l).$$\end{linenomath*} Let $\phi_{n-1,1}=(y_1, \ldots, y_n)$.  Now we have to produce $\phi_{1,1}$ which fits in the 
 sequence (\ref{phisequence}).
Consider $\phi_{1,1}=(x_1, x_2,\ldots, x_n)$. Since $\pi_1(\phi_{1,1})=\tau_{1,1}$, it follows that $x_{1}=l$.

\noindent\underline{\bf Case-1:} Assume that $j\leq i$. By the sequence \ref{phisequence}, we observe that for $1<l\leq i+1$, we have $x_{l}\leq l$. Further, note that $y_{l}=l+1$.

\noindent\underline{\bf Case-2:} Assume that $j>i$. As $x_{j+1}=1$, 
we can see that \begin{linenomath*} $$(x_1,\ldots, x_{j+1})\uparrow \ngeq (y_1, \ldots, y_{j+1})\uparrow$$ \end{linenomath*}(the ordering fails at $(i+1)^{th}$ position); here $\uparrow$ indicates that $x_i$'s are arranged in increasing order.
Hence, we conclude that there is no $\phi_{1,1}$ which fits in the sequence (\ref{phisequence}).
Thus, it follows $dim(H^0(X(v_{1, j}), \mathcal L_{\alpha_0})^T=1$ and this proves the lemma.
 \end{proof}

\begin{lemma}\label{generators2}
 For $k=1$ and $j=n-1$; or for $2\leq k\leq i$ and $2\leq j\leq n-k$; 
 \begin{enumerate}
 \item We have $dim(H^0(X(v_{k, j}), \mathcal L_{\alpha_0})^T=i+1$.
\item The section $P_{\Lambda}\in H^0(X(v_{k, j}), \mathcal L_{\alpha_0})$ is $T$-invariant if and only if 
$$\tau_{1,1}=l~\text{and}~\tau_{n-1, 1}=1, \ldots l-1, l+1, \ldots, n~\text{for}~l\in\{2, \ldots , i+2\}.$$
\end{enumerate}
\end{lemma}
\begin{proof}
\noindent\underline{\bf Case-1:} $k=1$ and $j=n-1$.

 First we give the sequences $v_{1, n-1}\geq \phi_{1,1}\geq \phi_{n-1, 1}$ as in (\ref{phisequence}).

For $2\leq l\leq i$, we take 
\begin{itemize}
 \item 
$\phi_{1,1}=(l, 2,3,\ldots, l-1, i+2, l+1, \ldots, i, i+3, \ldots, n, i+1, 1)$
 \item $\phi_{n-1,1}=(1,\ldots, l-1,l+1, \ldots, n, l)$ and 
\end{itemize}

and for $l=i+1$ we take
\begin{itemize}
 \item 
$\phi_{1,1}=(i+1, 2,3, \ldots, i, i+3, \ldots, n, i+2, 1)$
 \item $\phi_{n-1,1}=(1,\ldots, i,i+2, \ldots, n, i+1)$.
\end{itemize}

Clearly, $\tau_{1,1}=l$ with $l>i+2$ and $$\tau_{n-1,1}=(1,\ldots ,l-1, l+1, \ldots, n, l)$$
is not a Young tableau on $X(v_{1, n-1})$.
Now the proof is similar to that of Lemma \ref{generators1}. 

\noindent\underline{\bf Case-2:} For $2\leq k\leq i$ and $2\leq j\leq n-k$ the proof is same as in Case 1. This completes the proof of the lemma.
\end{proof}

\begin{lemma}\label{6.8} For $n-1\geq k>i+1$ and $1\leq j\leq n-k$; 
  
 \begin{enumerate}
 \item We have $dim(H^0(X(v_{k, j}), \mathcal L_{\alpha_0})^T=k$.
\item The section $p_{\Lambda}\in H^0(X(v_{k, j}), \mathcal L_{\alpha_0})$ is $T$-invariant 
if and only if 
$$\tau_{1,1}=l~\text{and}~\tau_{n-1, 1}=1, \ldots l-1, l+1, \ldots, n (~\text{for}~ l\in\{2, \ldots , k+1\}).$$  
\end{enumerate}

 \end{lemma}
\begin{proof}
 Proof is similar to the proofs of Lemma \ref{generators1} and \ref{generators2}. Further, note that $v_{n-1,1}=w_0$.
\end{proof}

\underline{\bf Case-A.2:} $i=n-1, n-2$.
In these cases, we have $w=s_{n-1}\cdots s_1$.
%\noindent\underline{\bf Notation:}
\begin{enumerate}
 \item ({\bf Minimal $v$}) Define $u_{0, n}:=s_{n-1} \cdots s_1$.
 \item ({\bf Extension $w$}) Define $u_{k, j}:=  s_j\cdots s_2 u_{k-1, n-(k-1)},$ if $1\leq k\leq n-2$ and $2\leq j\leq n-k$.
\end{enumerate}
\begin{lemma} \label{5.12}\

\begin{enumerate}
 \item 
 If $k=1$ and $2\leq j\leq n-2$, then $$dim(H^0(X(u_{1,j}), \mathcal L_{\alpha_0})^T)=1.$$
\item If $k=1$ and $j=n-1$; or if $k>1$, then $$dim(H^0(X(u_{k,j}), \mathcal L_{\alpha_0})^T)=n-1.$$ 
\end{enumerate}
 \end{lemma}
\begin{proof}
  The proof is similar to the proof in the case of $v_{k,j}$ by noticing that $$u_{0,n}=(n, 1,2,\ldots, n-1).$$
\end{proof}

\subsubsection{} \underline{{\bf Case-B:}} We take $v$ from Theorem \ref{santosh}(2).

%\noindent\underline{\bf Notation:}
For $1\leq i\leq n-1$, we define $v=s_i \cdots s_1s_{i+1}\cdots s_{n-1}$.
\begin{enumerate}
 \item ({\bf Minimal $v$}) Define $w_{0, n}:=s_{i}\cdots s_1s_{i+1}\cdots s_{n-1}$.
 \item ({\bf Extension $w$}) Define $w_{k, j}:= \begin{cases} s_j\cdots s_2 w_{k-1, n-(k-1)}, ~&\text{if}~ 1\leq k\leq i-1 ~ \text{and}~ 2\leq j\leq n-k\\
                   w_{i-1, n-i+1} ~~\text{if}~ k=i\\
                   s_j\cdots s_1 w_{k-1, n-(k-1)} ~&\text{if}~ k>i ~\text{and}~ 1\leq j\leq n-k.	
                  \end{cases}
 	 $
\end{enumerate}
\begin{lemma}\label{5.13}\

 \begin{enumerate}
  \item If $k=1$ and $2\leq j\leq n-2$, then $$dim(H^0(X(w_{k, j}), \mathcal L_{\alpha_0})^T)=1.$$
  \item For $k=1$ and $j={n-1}$; or $2\leq k\leq i$ and $2\leq j\leq n-k$, we have $$dim(H^0(X(w_{k, j}), \mathcal L_{\alpha_0})^T)=i.$$
  \item For $n>k>i$ and $1\leq j \leq n-k$, we have $$dim(H^0(X(w_{k, j}), \mathcal L_{\alpha_0})^T)=k.$$
 \end{enumerate}
\end{lemma}

\begin{proof}
 The proof follows by using similar arguments as in the case of $v_{k,j}$.
\end{proof}

\subsection{Proof of the main results}
In this section, we prove the main theorem by using the results from Sec.\ref{coordinatering}.

\begin{proof}[Proof of Theorem \ref{projectivenormality}]

By \cite[Theorem 3.14, p.76]{GITmumford}, we have \begin{linenomath*} $$(X_P(w)^{ss}_{T}(\mathcal L_{\alpha_0})\sslash T, \mathcal M) = Proj(\bigoplus_{m\geq 0}H^0(X_P(w), \mathcal L_{m\alpha_0})^T).$$ \end{linenomath*}

As in Sec. \ref{coordinatering}, we reduce the proof to the case of Schubert varieties in the flag variety $G/B$. Since we assume that $(X_P(w)^{ss}_{T}(\mathcal L_{\alpha_0})$ is nonempty, there exists a minimal $v \in W$ such that $w \geq v$ and $(X_P(w)^{ss}_{T}(\mathcal L_{\alpha_0})$ is nonempty. 

We first assume that $w$ is of the form $w=v_{k,l}$ for some $k, l$ (see Sec. \ref{coordinatering} for the notation). 
%Then we have the equality:
%$$H^0(X(v_{k, j}), \mathcal L_{m\alpha_0})=H^0(X_P(w), \mathcal L_{m\alpha_0}).$$
Let $$R(v_{k, j}):=\bigoplus_{m\geq 0}H^0(X(v_{k, j}), \mathcal L_{m\alpha_0})^T ~\text{and}~ R_m:=H^0(X(v_{k, j}), \mathcal L_{m\alpha_0})^T.$$

\noindent \underline{{\bf Claim:}} For a given $m$, we show that every element of $R_m$ is a product of $m$ elements of $R_1$.

%Remember that $m\alpha_0=m\omega_1+m\omega_{n-1}$.
We prove the claim by taking cases on $k$ and $j$.

\noindent\underline{\bf Case 1:} $k=0$ or $k=1$ and $2\leq j\leq n-2$.

For $k=0$ we have $w=v_{0,n}$. By Lemmas \ref{6.4} and \ref{generators1}, we have $dim(R_1)=1$ and so any element $p_{\Lambda}\in R_m$ is of the form $f^m$ for some $f\in R_1$. So, in this case the claim follows.

\noindent\underline{\bf Case 2:} $k=1$ and $j=n-1$; or $2\leq k\leq i+1$ and $2\leq j\leq n-k$.

By Lemma \ref{generators2}, we have $dim(R_1)=i+1$.
For a section $p_{\Lambda}\in R_m$, the associated Young tableau $\Lambda$ is a of shape $m\alpha_0$.
Since the Schubert varieties are projectively normal, we get $$p_{\Lambda}=p_{\Lambda_1}p_{\Lambda_2}\cdots p_{\Lambda_m}, ~\text{where}~ \Lambda_i \in H^0(X(v_{k, j}), \mathcal L_{\alpha_0})$$ and 
each  $\Lambda_i$ is of shape $\alpha_0$.
Note that the top $m$ boxes in the first column of $\Lambda$ are filed with $1$'s  and the rest filled with non decreasing numbers from the set  $\{2, \ldots, i+2\}.$ 

Notice that the last $m$ boxes of the first column determine the Young tableau $\Lambda$ such that $p_{\Lambda}$ is $T$-invariant. So, we conclude that 
$p_{\Lambda_i}$'s are in $R_1$ and hence the claim follows.
%every element of $R_m$ is the product of $m$ elements of $R_1$.

\noindent\underline{\bf Case 3:} $i+2\leq k\leq n-1$ and $1\leq j\leq n-k$.
This case is similar to that of Case 2 by using Lemma \ref{6.8}, but the only difference here is that the last $m$ boxes of the first column are filled with non-decreasing numbers from $\{2,\ldots, k+1\}$.
 
 Thus, we proved the claim in the case $w=v_{k, j}$. The proof of the claim for $w=u_{k,j}$ or $w=w_{k,j}$ is similar. 
 
 By the claim, the graded ring $R(w)$ is generated by $R_1$. Therefore, using \cite[Ex. 5.14, II]{hartshorne}, we conclude that the polarized variety $(X_P(w)^{ss}_{T}(\mathcal L_{\alpha_0})\sslash T, \mathcal M)$ is projectively normal .
\end{proof}

\begin{proof}[Proof of Corollary  \ref{quotient}] Recall that $$R_m:=H^0(X(w), \mathcal L_{m\alpha_0})^T.$$ Let $R_1=span_{\mathbb C}\{f_1,f_2,\ldots ,f_t\}$ with
$f_i$'s are standard monomials on $X(w)$. 

We first claim that the set $\{f_1,f_2,\ldots ,f_t\}$ is algebraically independent.
To prove this, it is enough to show that the set of all monomials in $f_i$'s are linearly independent. The monomials $f_i$'s are standard 
and observe that any monomial of degree $m$  in $f_i$'s is a standard monomial on $X(w)$ in the original Pl\"ucker coordinates. 
Using the fact that the standard monomials of degree $m$ on $X(w)$ form a basis of $H^0(X(w), \mathcal L_{m\alpha_0})$, we conclude the claim.
%set $\{f_1,f_2,\ldots ,f_t\}$ is algebraically independent. 

By Theorem \ref{projectivenormality}, any element of $R_m$ is a product of elements of $R_1$ for all $m$. Hence, there is an isomorphism of $\mathbb C$- algebras:
\begin{linenomath*}$$\bigoplus_{m\geq 0}R_m \simeq \mathbb C [R_1].$$\end{linenomath*}
Notice that all the generators of $R_1$ are of same degree in the Pl\"ucker coordinates.
Thus, the proof follows. 
%Hence, we conclude that the quotient $(X_P(w)^{ss}_{T}(\mathcal L_{\alpha_0})\sslash T$ is a projective space.
\end{proof}
\noindent We conclude this section by recovering Theorem 3.3 of \cite{CSS} for $G=SL(n, \mathbb C)$.
\begin{corollary}\

\begin{enumerate}
 \item  The polarized variety 
 $((G/P)^{ss}_{T}(\mathcal L_{\alpha_0})\sslash T, \mathcal M)$ is projectively normal.
 \item The quotient $(G/P)^{ss}_{T}(\mathcal L_{\alpha_0})\sslash T$ is a projective space of dimension $n-2$.
\end{enumerate}
\end{corollary}
\begin{proof}
As $w_0=v_{n-1,1}$ and $\pi(X(w_0))=G/P$, the result follows from Theorem \ref{projectivenormality} together with Corollary \ref{quotient}. 
\end{proof}

\section*{Acknowledgments}
The first author would like to thank Max Planck Institute for Mathematics  (Bonn) for the postdoctoral fellowship, and for providing very pleasant hospitality. We also thank the referees for their insightful comments and suggestions.
%We thank Sarjick Bakshi for informing us a minor error in Proposition \ref{semi-singular} in the preliminary version.

\end{document}